\documentclass[12pt]{amsart}
\usepackage{amsmath,amsxtra,amssymb,latexsym, amscd,amsthm,enumerate}
\usepackage{graphicx, tikz}
\topmargin = - 0.5cm
\numberwithin{equation}{section}

\newtheorem{thm}{\bf Theorem}[section]
\newtheorem{lem}[thm]{\bf Lemma}

\newtheorem{cor}[thm]{\bf Corollary}
\newtheorem{prop}[thm]{\bf Proposition}

\theoremstyle{definition}
\newtheorem{ex}[thm]{Example}
\newtheorem{rem}[thm]{Remark}
\newtheorem{quest}[thm]{Question}
\newtheorem*{thm*}{Theorem}

\def\NN{{\mathbb N}}
\def\mm{{\mathfrak m}}

\def\a{{\alpha}}
\def\b{{\beta}}

\DeclareMathOperator{\reg}{reg}

\begin{document}

\title{On the reduction numbers and\\ the Castelnuovo-Mumford regularity\\ of projective monomial curves}

\author{Tran Thi Gia Lam}
\address{Phu Yen University, Tuy Hoa, Vietnam}
\email{tranthigialam@pyu.edu.vn}


\begin{abstract}
This paper gives explicit formulas for the reduction number and the Castelnuovo-Mumford regularity of projective monomial curves.
\end{abstract}

\maketitle


\section*{Introduction}

Let $R$ be a Noetherian ring and $I$ an ideal in $R$. An ideal $J \subseteq I$ is called a {\em reduction} of $I$ if $JI^r = I^{r+1}$ for some $r \ge 0$. We denote by $r_J(I)$ the least number $r$ with this property and call it the {\em reduction number} of $I$ with respect to $J$. A reduction of $I$ is {\em minimal} if it does not contain any other reduction of $I$. The {\em reduction number} of $I$ is defined by
$$r(I) = \min\{r_J(I)|\ \text{$J$ is a minimal reduction of $I$}\}.$$
These notions were introduced by Northcott and Rees \cite{NR} and play an important role in the theory of integral closure and in the study of the blow-up rings, see e.g. \cite{Tr3,Va2}.

Let $R$ be a finitely generated standard graded algebra over a field $k$ and let $R_+$ be its maximal graded ideal. One calls $r(R_+)$ the reduction number of the algebra $R$ \cite{Tr4, Va1}. In this paper, we will investigate the reduction number of the coordinate ring of projective monomial curves, which are given parametrically by monomials of the same degree in two variables. Our investigation is motivated by the fact that in this case, the reduction number is a good approximation of the Castelnuovo-Mumford regularity $\reg(R)$ of $R$, which is a universal bound for many important invariants of $R$ \cite{Tr5}.  For instance, $\reg(R)$ is an upper bound for the maximal degree of the defining equations of $R$.  

Let $M$ be a set of monomials of degree $d$ in two variables $x,y$. Without restriction we may assume that $x^d, y^d \in M$. Let $R$ be the coordinate ring of the corresponding monomial curve. Then $R$ is isomorphic to the subalgebra $k[M]$ of $k[x,y]$ generated by the monomials of $M$.
Represent $R$ as a quotient of a polynomial ring $S$. Consider a minimal free resolution of $R$ over $S$:
$$0 \to F_p \to \cdots \to F_0 \to R \to 0$$
where $F_i$ is a graded free module over $S$, $i = 0,...,p$. Let $b_i$ be the maximal degree of the generators of $F_i$. Then the {\em Castelnuovo-Mumford regularity} of $R$ is the number
$$\reg(R) = \max\{b_i - i|\ i = 0,...,p\}.$$
Roughly speaking, the behavior of the graded components of $R$ will become regular in the degrees $\ge \reg(R)$. 
To compute $\reg(R)$ is usually a hard problem. By a result of Gruson, Lazarsfeld and Peskine \cite{GLP} we know that $\reg R \le d-|M|-2$. Finer bounds for $\reg(R)$ in terms of $M$ have been given in \cite{BGGM, HHS, Lv}.

It is well known that $\reg(R)$ is also the regularity of $R$ as a module over the ring $A = k[x^d,y^d]$, which is a  Noether normalization of $R$.  The minimal free resolution of $R$ over $A$ has the simple form
$$0 \to E_1 \to E_0 \to R \to 0,$$
where $E_i$ is a graded free module over $A$, $i = 0,1$. Let $c_i$ be the maximal degree of the generators of $F_i$. Then 
$$\reg(R) = \max\{c_0,c_1-1\}.$$

Let $Q = (x^d,y^d)$ be the ideal generated by $x^d, y^d$ in $R$. It is easy to see that $Q$ is a minimal reduction of $R_+$ and $r_Q(R_+) = c_0$. Hence, $r_Q(R_+)$ is a good approximation of $\reg(R)$. In particular, $\reg(R) = r_Q(R_+)$ if $R$ is a Cohen-Macaulay ring or a Buchsbaum ring \cite{Tr2}. It had been an open question whether $\reg(R) = r_Q(R_+)$ always holds until Hellus, Hoa and St\"uckrad found a counter-example in \cite{HHS}.  

The goal of this paper is to find {\em explicit formulas} for $r_Q(R_+)$ in terms of $M$. 
We will represent $M$ as a non-increasing sequence of non-negative integers
$$0=a_0\le a_1<a_2\le a_3<\ldots<a_{2m}\le a_{2m+1}=d $$
with $a_{2(i+1)}\ge a_{2i+1}+2$, $i = 0,...,m$, such that 
$$ M=\{x^{\alpha}y^{d-\alpha} \ | \ \alpha \in [a_{2i}, a_{2i+1}], i = 0,..., m\},$$
where $[a_{2i}, a_{2i+1}]$ denotes the set of integers $\a$, $a_{2i} \le \a \le a_{2i+1}$. 
We may view $[a_{2i}, a_{2i+1}]$ as an interval in the segment $[0,d]$.
Our main results give explicit formulas for $r_Q(R_+)$ and $\reg(R)$ in terms of $a_1,...,a_{2m+1}$ in the cases of Figure 1, where the given intervals are painted bold and the dots mean intervals of one point.

\begin{figure}[h]

\begin{tikzpicture}[scale=0.5]
\draw (-4,0.1) node {Case A:};
\draw [-] (0,0) -- (9,0);
\draw [line width=0.1cm] (4,0) -- (9,0);
\draw [line width=0.1cm] (0,0) -- (0.2,0);
\end{tikzpicture}
 
\begin{tikzpicture}[scale=0.5]
\draw (-4,0.1) node {Case B:};
\draw [-] (0,0) -- (9,0);
\draw [line width=0.1cm] (0,0) -- (0.2,0);
\draw [line width=0.1cm] (3,0) -- (6.5,0);
\draw [line width=0.1cm] (8.8,0) -- (9,0);
\end{tikzpicture}

\begin{tikzpicture}[scale=0.5]
\draw (-4,0.1) node {Case C:};
\draw [-] (0,0) -- (9,0);
\draw [line width=0.1cm] (0,0) -- (2,0);
\draw [line width=0.1cm] (6,0) -- (9,0);
\end{tikzpicture}

\begin{tikzpicture}[scale=0.5]
\draw (-4,0.1) node {Case D:};
\draw [-] (0,0) -- (9,0);
\draw [line width=0.1cm] (0,0) -- (0.2,0);
\draw [line width=0.1cm] (3,0) -- (4.5,0);
\draw [line width=0.1cm] (6,0) -- (9,0);
\end{tikzpicture}

\begin{tikzpicture}[scale=0.5]
\draw (-4,0.1) node {Case E:};
\draw [-] (0,0) -- (9,0);
\draw [line width=0.1cm] (0,0) -- (2,0);
\draw [line width=0.1cm] (4,0) -- (5.5,0);
\draw [line width=0.1cm] (7,0) -- (9,0);
\end{tikzpicture}
\caption{}
\end{figure}

Our results shows that the computation of $r_Q(R_+)$ depends very much on whether the two intervals on both ends consist of one point, i.e. $a_1 = 0$ or $a_{2m} = d$, or not. This was already observed in \cite{Tr1}, where the Cohen-Macaulayness of projective monomial curves was studied. Case C follows from results of Hellus, Hoa and St\"uckrad for smooth monomial curves \cite{HHS}. Our investigation can be served as a starting point for a systematic study of the Castelnuovo-Mumford regularity of projective monomial curves via the reduction number. 

The paper is organized as follows. The main part of the paper is divided in 5 sections. In Section 1 we prepare some properties of the minimal reduction $Q$ of $R_+$ and translate the computation of $r_Q(R_+)$ into a combinatorial problems on lattice points. Sections 2-6 will deal with the cases A--E, respectively.

\section{Preliminaries}

Throughout this paper, let $R = k[M]$ be the algebra over a field $k$ generated by a given set $M$ of monomials of degree $d > 0$ in two variables $x,y$, which contains $x^d,y^d$. If we set $\deg f = 1$ for all $f \in M$, then $R$ becomes a graded algebra. Suppose that $R = \oplus_{n \ge 0}R_n$.
Let $R_+ = \oplus_{n \ge 1}R_n$, which is the maximal graded ideal of $R$. 

\begin{lem} \label{reduct}
Let $Q$ be a homogeneous ideal of $R$ generated by elements of degree 1. Then $Q$ is a reduction of $R_+$ if and only if $R_n = Q_n$ for some $n \ge 1$, and $r_Q(R_+)$ is the least integer $n$ such that $R_{n+1} = Q_{n+1}$.
\end{lem}

\begin{proof}
Since $R$ is a standard graded algebra, i.e. it is generated by $R_1$, $(R_+)^n$ is generated by $R_n$.
Similary, since $Q$ is generated by $Q_1$, $Q^n$ is generated by $Q_n$. Therefore, $(R_+)^{n+1} = Q(R_+)^n$ if and only if $R_{n+1} = Q_{n+1}$. Hence, the conclusion follows from the definition of reduction and reduction number.
\end{proof}

The proof shows that Lemma \ref{reduct} holds for any standard graded algebra.

From now on, let $Q = (x^d, y^d)$ be the ideal generated by $x^d, y^d$ in $R$.

\begin{lem}
$Q$ is the unique minimal reduction of $R_+$ generated by monomials.
\end{lem}

\begin{proof}
Let $f$ be an arbitrary monomial of degree $r(d-1)+1$ in $R$, where $r$ is the number of monomials in $M$. 
Since $f$ is a product of $r(d-1)+1$ monomials in $M$, $f$ is divisible by $g^d$ for some $g \in M$. Since $g = x^ay^b$ for some $a,b \ge 0$, it follows that $f$ is divisible by $x^d$ or $y^d$ in $R$. So we conclude that $R_{r(d-1)+1} = Q_{r(d-1)+1}$. By Lemma \ref{reduct}, $Q$ is a reduction of $R_+$. Since $\dim R = 2$, any reduction of $R_+$ generated by two elements is a minimal reduction of $R_+$ \cite[Theorem 2 of p.~151 and Theorem 1 of p.~154]{NR}. Therefore, $Q$ is a minimal reduction of $R_+$.

If there were an other minimal reduction of $R_+$ generated by two monomials $f,g$, then $f, g \in M$ and $x^{nd}, y^{nd}$ must be divisible by $f$ or $g$ for some $n \ge 1$. From this it follows that $f, g$ must be powers of $x$ or $y$. Since $x^d,y^d$ are the only generators of $M$ having these forms, we conclude that $f,g$ are $x^d,y^d$. 
\end{proof}

Despite the uniqueness of $Q$ we may have $r_Q(R_+) > r(R_+)$. This can be seen as follows.

Let $I$ be the ideal generated by $M$ in $k[x,y]$. Let $\mm$ be the maximal graded ideal of $k[x,y]$.
Let $F(I) = \oplus I^n/\mm I^n$, which is the fiber ring of $I$. It is easy to see that $F(I) \cong R$. 
By \cite[Lemma 4.1]{Tr4}, there is a one-to-one correspondence between the minimal reductions of $I$ and the minimal reductions of $R_+$, which preserves the reduction number.
Huckaba showed that the reduction number of $I$ with respect to a minimal reduction $J$ depens on the choice of $J$  \cite[Example 3.1]{Hu}. His result leads to the following example.

\begin{ex} 
Let $I = (x^7,x^6y,x^2y^5,y^7) \subseteq k[x,y]$ and $J_1 = (x^7,y^7)$, $J_2 = (x^7,x^6y+y^7)$. By \cite[Example 3.1]{Hu}, $r_{J_1}(I) = 4 > r_{J_2}(I)$. Hence, $r_{J_1}(I) > r(I)$. Let $R = k[x^7,x^6y,x^2y^5,y^7]$ and $Q = (x^7,y^7) \subseteq R$. By  \cite[Lemma 4.1]{Tr4}, this implies $r_Q(R_+) > r(R_+)$. 
\end{ex}

On the other hands, there are many cases where $r_Q(R_+) = r(R_+)$.
In the following we say that $R$ is a Buchsbaum ring if the local ring of $R$ at the prime ideal $R_+$ is a Buchsbaum ring.
We refer the reader to \cite{SV} for the definition of Buchsbaum local rings. 
Notice that this class of rings contains all Cohen-Macaulay ring.

\begin{prop} \label{buchs}
Let $R$ be a Buchsbaum ring. Then
$$\reg(R) = r_Q(R_+) = r(R_+).$$
\end{prop}

\begin{proof}
This result follows from \cite[Corollary 3.5]{Tr2}. Note that the reduction number is defined there differently.
\end{proof}

We always have $\reg(R) \ge r_Q(R_+)$ \cite[Proposition 3.2]{Tr2}. There are examples showing that $\reg(R) > r_Q(R_+)$ even when the corresponding monomial curve is smooth. These examples are found first by Hellus, Hoa and St\"uckrad \cite[Example 3.2]{HHS}.

Now we will translate the computation of $r_Q(R_+)$ into a combinatorial problem of lattice points in $\Bbb N^2.$ For this purpose, we identify any monomial $x^\a y^\b$ with the point $(\a,\b)$.

For $n \ge 0$ let 
$$H_n = \{(\alpha,\beta) \in \NN^2|\ x^{\alpha}y^{\beta} \in R, \alpha + \beta = nd\}.$$
Then $R_n$ is the vector space generated by the monomials $x^{\alpha}y^{\beta}$, $(\alpha,\beta) \in H_n$.
Let $e_1=(d,0)$ and $e_2=(0,d)$. 
Then $Q_n$ is the vector space generated by the monomials $x^{\alpha}y^{\beta}$, $(\alpha,\beta) \in H_{n-1} + \{e_1,e_2\}$ (Minkowski sum). Using Lemma \ref{reduct}, we immediately get
$$r_Q(R_+) = \min\{n \in \NN|\ H_{n+1} = H_n + \{e_1,e_2\}\}.$$

This formula can be simplified by considering the first components of the elements of $H_n$.

\begin{lem} \label{interval}
Let $E_n=\{\alpha|\ (\alpha,nd-\alpha)\in H_n\}.$ Then
$$r_Q(R_+) = \min\{n \in \NN|\ E_{n+1} = E_n + \{0,d\}\}.$$
\end{lem} 

\begin{proof}
We have
$$H_{n+1} = \{(\a,(n+1)d-\a)|\  \a \in E_{n+1}\},$$
$$H_n + \{e_1,e_2\} = \{(\a,(n+1)d-\a)|\  \a \in E_n + \{0,d\}\}.$$
Therefore, $H_{n+1} = H_n + \{e_1,e_2\}$ if and only if $E_{n+1} = E_n + \{0,d\}$.
\end{proof}

We will keep the notations of this section throughout  this paper.

\section{Case A}
 
In this section we study the ring $R = k[M]$, where $M$ is represented by the union of a point and an interval (see Figure 2).
\bigskip

\begin{figure}[h]

\begin{tikzpicture}[scale=0.5]
\draw [-] (0,0) -- (9,0);
\draw [line width=0.1cm] (4,0) -- (9,0);
\draw [line width=0.1cm] (-0.1,0) -- (0.1,0);
\end{tikzpicture}
\caption{}
\end{figure}

\begin{thm} \label{1}
Let $1 < a < d$ be integers. Let $R = k[x^{\alpha}y^{d-\alpha} | \ \alpha \in \{0\} \cup [a,d]]$ and $Q = (x^d,y^d)$. 
Then $R$ is a Cohen-Macaulay ring and
$$\reg(R) = r_Q(R_+)  = \left\lceil \dfrac {d-1}{d-a}\right\rceil.$$
\end {thm}

\begin{proof} 
The Cohen-Macaulayness of $R$ follows from \cite[Theorem 3.5]{Tr1}. 
By Proposition \ref{buchs}, it remains to show that 
$r_Q(R_+)  =  \left\lceil \dfrac {d-1}{d-a}\right\rceil,$
which will be proved by using Lemma \ref{interval}.

Let $E_n := \{\a|\  x^\a y^{nd-\a} \in R\}$. It is easy to see that $E_n = \bigcup_{i=0}^{n}[ia, id]$. Hence
\begin{equation}
E_{n+1} = E_n \cup [(n+1)a,(n+1)d],
\end{equation}
\begin{align*}
E_n + \{0,d\} & = (E_n + 0)\cup (E_n +d)\\
&= E_n \cup \bigcup_{i = 0}^n [ia+d,(i+1)d].
\end{align*}
For $i = 0,...,n-1$, we have $[ia+d,(i+1)d] \subseteq [(i+1)a,(i+1)d] \subseteq E_n$. Hence
\begin{equation}
E_n + \{0,d\} = E_n \cup [na+d,(n+1)d].
\end{equation}

Note that 
$$[(n+1)a,(n+1)d] \setminus [na+d,(n+1)d] = [(n+1)a,na+d-1].$$
Then it follows from (2.1) and (2.2) that $E_{n+1} = E_n +\{0,d\}$ if and only if 
$$[(n+1)a,na+d-1] \subseteq E_n.$$

We have $na+d-1 \in E_n$ if and only if $na+d-1 \in [ia,id]$ for some $i = 0,...,n$, which implies $na+d+1 \le id \le nd$. Since $na < (n+1)a$, $[(n+1)a,na+d-1] \subseteq [na,nd]$. 
Therefore, $[(n+1)a,na+d-1] \subseteq E_n$ if and only if $[(n+1)a,na+d-1] \subseteq [na,nd]$. This condition is satisfied  
if and only if $na + d - 1 \le nd$. Thus,  $E_{n+1} = E_n +\{0,d\}$ if and only if $\dfrac{d-1}{d-a} \le n$. 
By Lemma \ref{interval}, this implies  $r_Q(R_+)  = \left\lceil \dfrac {d-1}{d-a}\right\rceil.$
\end{proof}

From Theorem \ref{1} we  immediately obtain the following result on monomial space curves.

\begin{cor}
Let $R = k[x^d,x^{d-1}y,x^{d-2}y^2,y^d],$ $d \ge 4$, and $Q = (x^d,y^d)$. Then $R$ is a Cohen-Macaulay ring and
$$\reg(R)  = r_Q(R_+) = \left\lceil \dfrac {d-1}{2}\right\rceil.$$
\end {cor}

\begin{rem}
Theorem \ref{1} also covers the case $R = k[x^{d-\alpha}y^\a | \ \alpha \in [0,a] \cup \{d\}]$ by permuting $x,y$.
\end{rem}

\section{Case B}

In this section we study the ring $R = k[M]$, where $M$ is represented by two points at both ends and an interval in the middle (see Figure 3). In this case, $R$ is always a Cohen-Macaulay ring \cite[Theorem 2.1]{Tr1}. By Proposition \ref{buchs}, this implies
$\reg(R) = r_Q(R_+).$
\bigskip

\begin{figure}[h]
\begin{tikzpicture}[scale=0.5]
\draw [-] (0,0) -- (9,0);
\draw [line width=0.1cm] (-0.1,0) -- (0.1,0);
\draw [line width=0.1cm] (3,0) -- (6.5,0);
\draw [line width=0.1cm] (8.9,0) -- (9.1,0);
\end{tikzpicture}
\caption{}
\end{figure}

If the middle interval consists of one point, we have the following simple formula for $r_Q(R_+)$.

\begin{prop}
Let $2 \le a < d$ be intergers such that $d \ge a + 2$. Let $R = k[x^d,x^ay^{d-a},y^d]$  and $Q = (x^d,y^d)$.
Then $R$ is a Cohen-Macaulay ring and
$$\reg(R) = r_Q(R_+) =  d/(a, d)-1,$$
where $(a,d)$ denotes the greatest common divisor of $a,d$.
\end {prop}

\begin{proof}
It is easy to see that $r_Q(R_+)$ is the least integer $n$ such that $(x^a y^{d-a})^{n+1}$ is divided by $x^d$ or, equivalently, by $y^d$. This number is clearly $d/(a, d)-1$. Actually, $R$ is defined by a binomial of degree $d/(a, d)-1$.
\end{proof}

If the middle interval consists of more than one point, the computation of $r_Q(R_+)$ becomes very complicated. However, we can give explicit formulas for $r_Q(R_+)$ in many cases. This is based on the following observation.

\begin{lem} \label{connect}
Let $0 < a < b$ be integers. For any integer $m \ge \dfrac{a -1}{b-a}$, $n \ge m$, and any number $c$, we have
$$\bigcup_{i=m}^n [ia+c,ib+c] = [ma+c,nb+c].$$
\end {lem}

\begin{proof}
We may assume that $n > m$.
By induction on $n$ we may assume that
$$\bigcup_{i=m}^{n-1} [ia+c,ib+c] = [ma+c,(n-1)b+c].$$
Then 
$$\bigcup_{i=m}^n [ia+c,ib+c] = [ma+c,(n-1)b+c] \cup [na +c, nb+c].$$
Since $n-1 \ge \dfrac{a -1}{b-a}$, $(n-1)b + c \ge na + c$. Hence
$$[ma +c, (n-1)b+c] \cup [na +c, nb+c] = [ma +c, nb+c].$$ 
\end{proof}

\begin{thm} \label{3}
Let $2 \le a < b \le d-2$ be integers.
Let $R = k[x^{\alpha}y^{d-\alpha} |  \alpha \in \{0,d\} \cup [a,b]]$ and $Q = (x^d,y^d)$. Then $R$  is a Cohen-Macaulay ring and
\begin{enumerate}[\rm (1)]
\item $\reg(R) = r_Q(R_+) = \left\lceil\dfrac {a+d-1}{b}\right\rceil$ if\; $b\ge 2a-1,$  
\item $\reg(R) = r_Q(R_+) =  \left\lceil\dfrac {2a+d-1}{b}\right\rceil$ if\; $\dfrac {3a-1}{2}\le b < 2a -1.$
\end{enumerate}
\end{thm}

\begin{proof}
We only need to prove the formulas for $r_Q(R_+)$. For $i,j \ge 0$, let
$$I_{i,j} := [ia +jd, ib+jd].$$
Let $E_n := \{\a|\  x^\a y^{nd-\a} \in R\}$. It is easily seen that
\begin{equation*}
E_n  = \bigcup_{i+j\le n}(i[a,b]+jd) = \bigcup_{i+j\le n}I_{i,j},
\end{equation*}
\begin{equation*}
E_{n+1} = \bigcup_{i+j\le n+1}I_{i,j} = E_n \cup \bigcup_{i+j= n+1}I_{i,j}  = E_n \cup \bigcup_{i=0}^{n+1}I_{i,n+1-i}.
\end{equation*}
Notice that $I_{i,j} + 0 = I_{i,j}$ and $I_{i,j} + d= I_{i,j+1}$. Then
\begin{equation*}
E_n + \{0,d\} = \bigcup_{i+j\le n}(I_{i,j} \cup I_{i,j+1}) = E_n \cup \bigcup_{i=0}^nI_{i,n+1-i}.
\end{equation*}
Therefore, $E_{n+1} = E_n + \{0,d\}$ if and only if 
$$I_{n+1,0} = [(n+1)a,(n+1)b] \subseteq E_n + \{0,d\}.$$

We may assume that $n \ge 1$. Then
\begin{equation}
\begin{split}
E_n + \{0,d\} & = \bigcup_{i+j\le n}(I_{i,j} \cup I_{i,j+1}) = \bigcup_{i=0}^n \bigcup_{j = 0}^{n+1-i}I_{i,j}\\
& = \bigcup_{j = 0}^{n+1} I_{0,j} \cup \bigcup_{i = 1}^{n} I_{i,0} \cup \bigcup_{j = 1}^{n} \bigcup_{i=1}^{n+1-j}I_{i,j}.
\end{split}
\end{equation}

If $b\ge 2a-1,$ then $\left\lceil \dfrac{a-1}{b-a}\right\rceil = 1$. By Lemma \ref{connect}, 
\begin{align*} 
\bigcup_{i = 1}^{n} I_{i,0} & = [a,nb],\\
\bigcup_{i=1}^{n+1-j}I_{i,j} & =  [a+jd, (n+1-j)b+jd].
\end{align*}
Hence we may rewrite (3.1) as
\begin{equation}
E_n + \{0,d\} = \{0,d,...,(n+1)d\} \cup [a,nb] \cup \bigcup_{j = 1}^{n} [a+jd, (n+1-j)b+jd].
\end{equation}

For any element $p \in \{2d,...,(n+1)d\} \cup \displaystyle \bigcup_{j = 1}^{n} [a+jd, (n+1-j)b+jd]$, we have $p \ge a+d$. If $p \in [(n+1)a,(n+1)b]$, then $p \le (n+1)b < nb+d$, hence $p \in [a+d,nb+d].$ 
From this it follows that
$$[(n+1)a,(n+1)b] \cap \big(\{2d,...,(n+1)d\} \cup \bigcup_{j = 1}^{n} [a+jd, (n+1-j)b+jd]\big) \subseteq [a+d,nb+d].$$
By (3.2), this implies that $[(n+1)a,(n+1)b] \subseteq E_n + \{0,d\}$ if and only if 
\begin{equation*}
[(n+1)a,(n+1)b] \subseteq  \{d\} \cup [a,nb] \cup [a+d,nb + d].
\end{equation*}
Since $a < (n+1)a \le (n-1)a + b+1 \le nb$, we have 
$$[(n+1)a,(n+1)b] \setminus [a,nb] = [nb+1,(n+1)b].$$
Therefore, $[(n+1)a,(n+1)b] \subseteq E_n + \{0,d\}$ if and only if
\begin{equation}
[nb+1,(n+1)b] \subseteq  \{d\} \cup [a+d,nb + d].  
\end{equation}

Assume that  (3.3) satisfied.
Since the interval $[nb+1,(n+1)b]$ contains $b > 2$ points, $(n+1)b \in [a+d,nb + d]$.
If $nb+1 < a+d$, then (3.3) implies $nb+1 = d = a+d-1$. Hence, we get $a=1$, a contradiction.
So we must have $nb+1 \ge a+d$. Conversely, if $nb+1 \ge a+d$, then 
$[nb+1,(n+1)b] \subseteq  [a+d,nb + d].$ Therefore, (3.3) is satisfied if and only if $nb+1 \ge a+d$, which is equivalent to $n \ge \dfrac{a+d-1}{b}$.
 
Summing up, we have $E_{n+1} = E_n + \{0,d\}$ if and only if $n \ge \dfrac{a+d-1}{b}$. By  Lemma \ref{interval}, this proves statement (1) of Theorem \ref{3}.
\par

If $\dfrac {3a-1}{2}\le b < 2a -1,$ then 
$1 <  \dfrac{a-1}{b-a} \le 2$. Hence
$\left\lceil \dfrac{a-1}{b-a}\right\rceil = 2$. 
By Lemma \ref{connect}, 
\begin{align*} 
\bigcup_{i = 2}^{n} I_{i,0} & = [2a,nb],\\
\bigcup_{i=2}^{n+1-j}I_{i,j} & =  [2a+jd, (n+1-j)b+jd].
\end{align*}
Hence we may rewrite (3.1) as
\begin{multline}
E_n + \{0,d\}  =\\ \{0,d,...,(n+1)d\} \cup [a,b] \cup [2a,nb] \cup \bigcup_{j=1}^n [a+jd,b+jd] \cup 
\bigcup_{j=1}^n[2a+jd, (n+1-j)b+jd].
\end{multline}

Recall that $E_{n+1} = E_n + \{0,d\}$ if and only if 
$$[(n+1)a,(n+1)b] \subseteq E_n + \{0,d\}.$$
Let $p$ be an element of $[(n+1)a,(n+1)b]$ such that
$$p \in \displaystyle \{3d,...,(n+1)d\} \cup \bigcup_{j=2}^n [a+jd,b+jd] \cup \bigcup_{j = 1}^{n} [2a +jd, (n+1-j)b+jd].$$
Then $p \ge 2a+d$. Hence $p \in [2a+d,nb+d]$ because $p \le (n+1)b < nb+d$.
Thus, 
\begin{multline*}
[(n+1)a,(n+1)b] \cap \big(\{3d,...,(n+1)d\} \cup \bigcup_{j=2}^n [a+jd,b+jd] \cup \bigcup_{j = 1}^{n} [2a +jd, (n+1-j)b+jd]\big)\\
\subseteq [2a+d,nb+d].
\end{multline*}
Since $b < 2a-1 < (n+1)a$, $[(n+1)a,(n+1)b] \cap [a,b] = \emptyset$.
Therefore, using (3.4) we have $[(n+1)a,(n+1)b] \subseteq E_n + \{0,d\}$ if and only if
$$[(n+1)a,(n+1)b]  \subseteq \{d,2d\} \cup [2a,nb] \cup [a+d,b+d] \cup [2a + d, nb+d].$$
Since $2a \le (n+1)a \le (n-2)a + 2b + 1 \le nb$, we have 
$$[(n+1)a,(n+1)b] \setminus [2a,nb] = [nb+1,(n+1)b].$$
Therefore, $[(n+1)a,(n+1)b] \subseteq E_n + \{0,d\}$ if and only if 
\begin{equation}
[nb+1,(n+1)b] \subseteq \{d,2d\} \cup [a+d,b+d] \cup [2a + d, nb+d].
\end{equation}

Note that $[nb+1,(n+1)b]$ has $b$ points and $\{d,2d\} \cup [a+d,b+d]$ has $b-a+3$ points. Since $2a - 1 > b \ge 3$, we have $a \ge 3$, which implies $b-a+3 \le b$. If
$[nb+1,(n+1)b] \subseteq  \{d,2d\} \cup [a+d,b+d]$, we would get $[nb+1,(n+1)b] = \{d,2d\} \cup [a+d,b+d]$, which is a contradiction because $\{d,2d\} \cup [a+d,b+d]$ is not an interval of integral points. Thus,
$[nb+1,(n+1)b] \not\subseteq  \{d,2d\} \cup [a+d,b+d]$.

Assume that (3.5) is satisfied. Then $[nb+1,(n+1)b] \cap [2a+d,nb+d] \neq \emptyset$. 
Hence $(n+1)b \in [2a + d, nb+d]$.
If $nb+1 < 2a+d$, then $2a+d-1 \in \{2d\} \cup [a+d,b+d]$.
Since $b+d < 2a+d-1$, this implies $2a+d-1 = 2d$. Hence, $2a-1 = d$.
Since $b+d+1 < 2d$, we also have $b+d+1 \not\in \{2d\} \cup [2a+d,nb+d]$. 
From this it follows that $b+d+1 < nb+1$. 
Hence, $nb+1\in \{2d\} \cup [2a + d, nb+d]$.
Therefore, $nb+1 = 2d$, which implies $n \ge 3$ because $b \le d-2$.
Since $2b \ge 3a-1$,  $nb + 1 \ge 3a + b \ge 4a+1 > 2d$, a contradiction. 
So we must have $nb+1 \ge 2a+d$. 
Conversely, if $nb+1 \ge 2a+d$, then $[nb+1,(n+1)b] \subseteq [2a + d, nb+d].$ 
So we have shown that (3.5) is satisfied if and only if $nb+1 \ge 2a+d$, which is equivalent to 
$n \ge \dfrac {2a+d-1}{b}$. 

Summing up, we have
$E_{n+1} = E_n + \{0,d\}$ if and only if $n \ge \dfrac {2a+d-1}{b}$. By  Lemma \ref{interval}, this proves statement (2) of Theorem \ref{3}. 
\end{proof}

From Theorem \ref{3} we immediately obtain the following results on monomial space curves.

\begin{cor} 
Let $R = k[x^d,x^3y^{d-3},x^2y^{d-2},y^d]$, $d \ge 5$, and $Q = (x^d,y^d)$. Then $R$ is a Cohen-Macaulay ring and
$$\reg(R) = r_Q(R_+) = \left\lceil\dfrac {d+1}{3}\right\rceil.$$
\end{cor}

\begin{cor} 
Let $R = k[x^d,x^4y^{d-4},x^3y^{d-3},y^d]$, $d \ge 6$, and $Q = (x^d,y^d)$. Then $R$ is a Cohen-Macaulay ring and
$$\reg(R) = r_Q(R_+) = \left\lceil\dfrac {d+5}{4}\right\rceil.$$
\end{cor}

The proof of Theorem \ref{3} leads to the following problem.

\begin{quest} 
Let $2 \le a < b \le d-2$ be integers and $r = \left\lceil \dfrac{a-1}{b-a}\right\rceil$.
Let $R = k[x^{\alpha}y^{d-\alpha} |  \alpha \in \{0,d\} \cup [a,b]]$ and $Q = (x^d,y^d)$. Is it true that
$$\reg(R) = r_Q(R_+) = \left\lceil\dfrac {ra+d-1}{b}\right\rceil ?$$
\end{quest} 

\section{Case C}

In this section we study the ring $R = k[M]$, where $M$ is represented by the union of two integral intervals which consist of more than one point (see Figure 4). 
 \bigskip

\begin{figure}[h]
\begin{tikzpicture}[scale=0.5]
\draw (-4,0.1) node {Case C:};
\draw [-] (0,0) -- (9,0);
\draw [line width=0.1cm] (0,0) -- (2,0);
\draw [line width=0.1cm] (6,0) -- (9,0);
\end{tikzpicture}
\caption{}
\end{figure}

It is known that a projective monomial curve parametrized by $M$ is smooth if and only if the two intervals at both ends of $M$ consist of one point. Hellus, Hoa and St\"uckrad \cite{HHS} did give explicit formulas for the reduction number of a few sporadic smooth monomial curves. The following result is an easy consequence of \cite[Proposition 3.4]{HHS}. 

\begin{prop} \label{2}
Let $1 \le a < b < d$ be intergers such that $b \ge a+2$. Let $R =  k[x^{\alpha}y^{d-\alpha} | \ \alpha \in [0, a] \cup [b,d]]$ and $Q = (x^d,y^d)$. Then
\[
\reg(R) = r_Q(R_+)= \begin{cases}
&\ \ \ \left\lceil\dfrac {b-1}{a}\right\rceil \ \ \ \ \ \mbox{ if }\ d-b \ge a,\\
&\left\lceil\dfrac{d-a-1}{d-b}\right\rceil \quad \mbox{ if }\ d - b< a.
\end{cases}
\]
\end {prop}

\begin{proof} 
By symmetry, we may assume that $d-b \ge a$. Then $R$ satisfies the assumptions of \cite[Proposition 3.4]{HHS}.
From this it follows that
$$\reg(R) = r_Q(R_+) =  \left\lfloor \dfrac {b-a-2}{a} \right\rfloor + 2.$$
It is easy to check that $\left\lfloor \dfrac {b-a-2}{a} \right\rfloor + 2 = \left\lceil\dfrac {b-1}{a}\right\rceil$.
\end{proof}

From Proposition \ref{2} we immediately obtain the following result on monomial space curves.

\begin{cor} \cite[Theorem 3.6(i)]{BGM}
Let $R = k[x^d,x^{d-1}y,xy^{d-1},y^d],$ $d \ge 4$, and $Q = (x^d,y^d)$. Then 
$$\reg(R) = r_Q(R_+) = d-2.$$
\end {cor}

\begin{rem}
Under the assumption of Theorem \ref{2}, $R$ is never a Cohen-Macaulay ring. This fact follows from \cite[Theorem 4.7]{Tr1}. 
\end{rem}

There is the following interesting relationship between the Buchsbaumness of $R$ and the reduction number.

\begin{prop} \label{reg 2}
Let $1 \le a < b < d$ be intergers such that $b \ge a+2$. Let $R =  k[x^{\alpha}y^{d-\alpha} | \ \alpha \in [0, a] \cup [b,d]].$ Then the following conditions are equivalent:
\begin{enumerate}[\rm (1)]
\item $R$ is a Buchsbaum ring,
\item $2a+1 \ge b$ and $a +d+1 \ge 2b$,
\item  $r_Q(R_+) = 2$.
\end{enumerate}
\end{prop}

\begin{proof}
By \cite[Theorem 4.3]{Tr1}, $R$ is a Buchsbaum ring if and only if $2a +1\ge b$ and $a +d +1\ge 2b$.\footnote{Actually, this is the case of \cite[Corollary 4.3]{Tr1}, which contains some misprints.}
By symmetry we may assume, without loss of generality, that $d-b \ge a$. 
Then $r_Q(R_+) = \left\lceil\dfrac {b-1}{a}\right\rceil$ by Proposition \ref{2}. 
It is clear that $\left\lceil\dfrac {b-1}{a}\right\rceil = 2$ if and only if $2a +1\ge b$, which implies
$a+d+1 \ge 2a+b+1 \ge 2b$. These facts show that (1)-(3) are equivalent. 
\end{proof}

\begin{cor}
Let $R = k[x^d,x^{d-1}y,xy^{d-1},y^d],$ $d \ge 4$, and   $Q = (x^d,y^d)$. Then $R$ is a Buchsbaum ring if and only if $d=4$. In this case,
$$\reg(R) = r_Q(R_+) = 2.$$
\end {cor}

\section{Case D}

In this section we study the ring $R = k[M]$, when $M$ is represented by one point at one end and two intervals (see Figure 5).
\bigskip

\begin{figure}[h]
\begin{tikzpicture}[scale=0.5]
\draw [-] (0,0) -- (9,0);
\draw [line width=0.1cm] (-0.1,0) -- (0.1,0);
\draw [line width=0.1cm] (3,0) -- (4.5,0);
\draw [line width=0.1cm] (6,0) -- (9,0);
\end{tikzpicture}
\caption{}
\end{figure}

The following result gives a large class of monomial curves of this case for which we can compute the reduction number.

\begin{thm} \label{5}
Let $1 < a \le b < c < d$ be integers. Let $R = k[x^{\alpha}y^{d-\alpha} | \ \alpha \in \{0\} \cup[a,b] \cup [c,d]]$ and $Q = (x^d,y^d)$.
Assume that $c \le 2a$ and $2b \le d$. Then $R$ is a Cohen-Macaulay ring and
$$\reg(R) = r_Q(R_+)  = \left\lceil \dfrac {a-1}{d-c}\right\rceil+1.$$
\end {thm}

\begin{proof} 
By \cite[Theorem 3.5]{Tr1}, $R$ is a Cohen-Macaulay ring if 
$$\left\lfloor\dfrac {a}{c-a}\right\rfloor > 
\min\left\{\left\lfloor \dfrac {c-b-2}{d-c}\right\rfloor,\left\lfloor \dfrac {b-1}{d-b}\right\rfloor\right\}.$$
This condition is satisfied because
$$\dfrac{a}{c-a} \ge \dfrac{a}{2a-a} = 1 \ge \dfrac{b}{d-b} > \dfrac{b-1}{d-b}.$$ 

By Lemma \ref{buchs}, it remains to show that 
$r_Q(R_+)  = \left\lceil \dfrac {a-1}{d-c}\right\rceil+1,$
which will be proved by using Lemma \ref{interval}.
For $n \ge 0$ let $E_n := \{\a|\  x^\a y^{nd-\a} \in R\}$.
We have
$$E_n= \bigcup_{i+j\le n}(i[a,b]+j[c,d]) = \bigcup_{i+j\le n}[ia +jc, ib+jd].$$
From this it follows that
\begin{align*}
E_{n+1} & =  E_n \cup \bigcup_{i+j =n+1} [ia +jc, ib+jd]\\
& = E_n \cup \bigcup_{i =0}^{n+1} [ia +(n+1-i)c, ib+(n+1-i)d].\\
E_n + \{0, d\} & = E_n \cup \bigcup_{i+j\le n}[ia +jc+d, ib+(j+1)d]\\
& = E_n \cup \bigcup_{i+j= n}[ia +jc+d, ib+(j+1)d]\\
& = E_n \cup \bigcup_{i=0}^n[ia +(n-i)c+d, ib+(n+1-i)d].
\end{align*}
We have $2[a,b] = [2a,2b] \subseteq [c,d]$ because $c \le 2a$ and $2b \le d$.
Therefore, $i[a,b] \subset E_{i-1}$ for $i \ge 2$. Hence, 
\begin{align*}
[ia +(n-i)c+d, ib+(n+1-i)d] & \subset [ia +(n+1-i)c, ib+(n+1-i)d]\\
& = i[a,b] + (n+1-i)[c,d] \subset E_n
\end{align*}
for $i \ge 2$.
From this it follows that
\begin{align}
E_{n+1} & = E_n \cup [(n+1)c,(n+1)d] \cup [a + nc,b+nd],\\
E_n + \{0, d\} & = E_n \cup [nc+d,(n+1)d] \cup [a + (n-1)c+ d, b +nd].
\end{align}

By (5.1) we have $E_{n+1}=E_n + \{0, d\}$ if and only if  
$$[(n+1)c,(n+1)d] \cup [a + nc,b+nd] \subseteq E_n + \{0, d\}.$$
Let $p$ be an arbitrary element of $[a + nc,b+nd] \cap E_n$. Then $p \ge a+nc > nc$ and $p \le nd$. Hence, 
\begin{equation}
[a + nc,b+nd] \cap E_n \subseteq [nc,nd], 
\end{equation}
Using (5.2) we have $[a + nc,b+nd] \subseteq E_n + \{0, d\}$ if and only if
$$[a + nc,b+nd] \subseteq [nc,nd] \cup [nc+d,(n+1)d] \cup [a + (n-1)c+ d,b +nd].$$
It is easy to check that 
$$[a + nc,b+nd] \cap [nc+d,(n+1)d] \subseteq [nc+d,b +nd] \subseteq [a + (n-1)c+ d, b +nd].$$
Therefore, $[a + nc,b+nd] \subseteq E_n + \{0, d\}$ if and only if
$$[a + nc,b+nd] \subseteq [nc,nd] \cup [a + (n-1)c+ d,b +nd].$$
The last condition is satisfied if and only if
$$[nc,nd] \cup [a + (n-1)c+ d, b +nd] = [nc,b+nd],$$
which is equivalent to $nd \ge a + (n-1)c+ d - 1$.

Similarly as above, we have $[(n+1)c,(n+1)d] \subseteq E_n + \{0, d\}$ if and only if
$$[(n+1)c,(n+1)d] \subseteq  [nc,nd] \cup [nc+d,(n+1)d] \cup [a + (n-1)c+ d,b +nd].$$
If $nd \ge a + (n-1)c+ d - 1$, then 
$b + nd \ge nc+d-1$ because $a + b \ge 2a \ge c$. 
From this it follows that
\begin{align*}
[nc,nd] \cup [nc+d,(n+1)d] \cup [a + (n-1)c+ d,b +nd] & = [nc,(n+1)d]\\
& \supseteq [(n+1)c,(n+1)d].
\end{align*}
Therefore, $[(n+1)c,(n+1)d] \subseteq E_n + \{0, d\}$ if $nd \ge a + (n-1)c+ d - 1$, which is equivalent to
$n \ge \dfrac {d - c + a - 1}{d-c}$.

Summing up, we have $E_{n+1} = E_n + \{0, d\}$ if and only if $n \ge \dfrac {d - c + a -1}{d-c}$.
By Lemma \ref{interval}, this implies $r_Q(R_+) = \left\lceil \dfrac {d-c+a-1}{d-c}\right\rceil = \left\lceil \dfrac {a-1}{d-c}\right\rceil +1.$
\end{proof}

From Theorem \ref{5} we immediately obtain the following consequence on monomial space curves.

\begin{cor} 
Let $R = k[x^d,x^{d-1}y,x^ay^{d-a},y^d]$,  $d \ge 5$, $a = \left\lceil \dfrac {d-1}{2}\right\rceil$, and $Q = (x^d,y^d)$. Then 
$R$ is a Cohen-Macaulay ring and
$$\reg(R) = r_Q(R_+) = \left\lceil\dfrac {d-1}{2}\right\rceil.$$
\end{cor}


\section{Case E}

In this section we study the ring $R = k[M]$ in the case $M$ is represented by the union of three intervals, where the two intervals at both ends consist of more than one point (see Figure 6).  
\bigskip

\begin{figure}[h]
 \begin{tikzpicture}[scale=0.5]
\draw [-] (0,0) -- (9,0);
\draw [line width=0.1cm] (0,0) -- (2,0);
\draw [line width=0.1cm] (4,0) -- (5.5,0);
\draw [line width=0.1cm] (7,0) -- (9,0);
\end{tikzpicture}
\caption{}
\end{figure}

Using \cite[Proposition 3.4]{HHS} we obtain the following partial result. 

\begin{prop} \label{6}
Let $1 \le a < b \le c < e < d$ be intergers such that $d-e \le a$ and $b - a\ge e-c$. Let $R =  k[x^{\alpha}y^{d-\alpha} | \ \alpha \in [0, a] \cup [b,c] \cup [b,d]]$ and $Q = (x^d,y^d)$. Then
\[
\reg(R) = r_Q(R_+)= \left\lceil\dfrac {b-1}{a}\right\rceil.
\]
\end {prop}

\begin{proof}
Actually, \cite[Proposition 3.4]{HHS} gives the formula 
$$\reg(R) = r_Q(R_+)=  \left\lfloor \dfrac {b-a-2}{a} \right\rfloor + 2.$$
It is easy to check that $\left\lfloor \dfrac {b-a-2}{a} \right\rfloor + 2 = \left\lceil\dfrac {b-1}{a}\right\rceil$.
\end{proof} 

The following result gives another large class of monomial curves for which we can compute the reduction number.

\begin{thm} \label{6}
Let $1 \le a < b \le c < e < d$ be intergers such that $e \le 2b$ and $2c \le a+d$. Let $R =  k[x^{\alpha}y^{d-\alpha} | \ \alpha \in [0, a] \cup [b,c] \cup [b,d]]$ and $Q = (x^d,y^d)$. Then
\[
r_Q(R_+)= \max\left\{\left\lceil\dfrac {b-1}{a}\right\rceil, \left\lceil\dfrac {e-c+a-1}{a}\right\rceil, \left\lceil\dfrac {d-c-1}{d-e}\right\rceil, \left\lceil\dfrac {d-e+b-a-1}{d-e}\right\rceil \right\}.
\]
\end {thm}

\begin{proof} 
For $n \ge 0$ let $E_n := \{\a|\  x^\a y^{nd-\a} \in R\}$ and 
$F_n  := \{\a|\  x^\a y^{nd-\a} \in S\},$
where $S = k[x^{\alpha}y^{d-\alpha} | \ \alpha \in [0, a] \cup [e,d]]$.
It is easy to see that 
$$E_n = \bigcup_{i=0}^n (F_{n-i}+i[b,c]).$$ 
The assumption $e \le 2b$ and $2c \le a+d$ implies 
$2[b,c] = [2b,2c] \subseteq [e,a+d] = [0,a] + [e,d] \subseteq F_2$.
From this it follows that 
$$E_n = F_n \cup (F_{n-1}+[b,c])$$
for $n \ge 1$. Therefore,  
\begin{align}
E_{n+1}  & = F_{n+1} \cup (F_n+[b,c]),\\
E_n + \{0, d\} & = (F_n + \{0, d\}) \cup   (F_{n-1} + [b,c] + \{0, d\}).
\end{align}

We have
$$F_n= \bigcup_{i+j=n}(i[0,a]+j[e,d]) = \bigcup_{i+j=n}[je, ia+jd]=\bigcup_{j=0}^{n}[je, (n-j)a+jd]. $$
From this it follows that 
\begin{align}
F_{n+1}  & = [0,(n+1)a]  \cup \bigcup_{j=1}^n [je, (n+1-j)a+jd] \cup [(n+1)e, (n+1)d],\\
F_n + [b,c] & = [b,na+c]  \cup \bigcup_{j=1}^{n-1} [b+je, c+ (n-j)a+jd] \cup [b+ne, c+nd].
\end{align}
Note that $F_n + \{0, d\} = (F_n + 0)\cup (F_n +d)$. Then
\begin{equation}
F_n + \{0, d\} = \bigcup_{j=0}^{n}[je, (n-j)a+jd] \cup \bigcup_{j=0}^{n}[je+d, (n-j)a+(j+1)d].
\end{equation}
For $j = 1,...,n$, we have $(n-j)a+jd \ge (j-1)e+d$, hence
$$[je, (n-j)a+jd] \cup [(j-1)e+d, (n-j+1)a+jd] = [je, (n-j+1)a+jd].$$
Therefore, we may rewrite (6.5) as
\begin{equation}
F_n + \{0, d\}  = [0,na] \cup \bigcup_{j=1}^{n}[je, (n+1-j)a+jd] \cup [ne+d,(n+1)d],
\end{equation}
From this it follows that
\begin{multline}
F_{n-1}+ [b,c] + \{0,d\} =\\
[b,(n-1)a+c] \cup \bigcup_{j=1}^{n-1}[b+je, (n-j)a+c+jd] \cup [b+(n-1)e+d,c+nd].
\end{multline}

Comparing (6.3) and (6.4) with (6.6) and (6.7) we have
\begin{align*}
& F_{n+1}  \setminus (F_n + \{0, d\}) \subseteq ([0,(n+1)a] \cup [(n+1)e, (n+1)d]),\\
& (F_n+[b,c]) \setminus (F_{n-1}+ [b,c] + \{0, d\}) \subseteq [b,na+c] \cup [b+ne, c+nd].
\end{align*}
From (6.1) and (6.2) it follows that $E_{n+1} = E_n + \{0, d\}$
if and only if  
\begin{equation}
[0,(n+1)a] \cup [(n+1)e, (n+1)d] \cup [b,na+c] \cup [ne+b,nd+c]\subseteq E_n + \{0, d\}.
\end{equation}

Now we are going to check when $[0,(n+1)a] \cup [b,na+c]$ is contained in $E_n + \{0, d\}.$
Let $p$ is an element of $[0,(n+1)a]$. 
If $p \in \displaystyle \bigcup_{j=1}^n[je, (n+1-j)a+jd]$ or $p \in [ne+d,(n+1)d]$, then $p \ge e$, which implies 
$p \in [e, na+d]$ because $p \le (n+1)a < na +d$. 
Therefore, using (6.6) we have
\begin{equation}
[0,(n+1)a] \cap (F_n + \{0, d\}) \subseteq [0,na] \cup [e, na+d].
\end{equation}
Similarly, if $p \in  \displaystyle \bigcup_{j=1}^{n-1}[b+je, (n-j)a+c+jd]$ or if $p \in [b+(n-1)e+d,c+nd]$, then $p \in [e, na+d]$. Therefore, using (6.7) we have
\begin{equation}
[0,(n+1)a] \cap (F_{n-1}+[b,c] + \{0, d\}) \subseteq [b,(n-1)a+c] \cup [e, na+d].
\end{equation}
By (6.2), it follows that
$[0,(n+1)a] \subseteq F_n + \{0, d\}$ if and only if 
\begin{equation}
[0, (n+1)a] \subseteq [0, na] \cup [b, (n-1)a + c] \cup [e,na+d].
\end{equation}
Similarly, we can show that $[b,na+c] \subseteq F_n + \{0, d\}$ if and only if 
\begin{equation}
[b,na+c] \subseteq [0, na] \cup [b, (n-1)a + c] \cup [e,na+d].
\end{equation}
It is easily seen that if (6.11) is satisfied, then $na+1 \ge b$ and that if (6.12) is satisfied, then $(n-1)a+c+ 1 \ge e$. Conversely, if $na+1 \ge b$ and $(n-1)a+c+ 1 \ge e$, then
$$[0, na] \cup [b, (n-1)a + c] \cup [e,na+d] = [0,na+d] \supseteq [0, (n+1)a] \cup [b,na+c],$$
hence (6.11) and (6.12) are satisfied. Thus, $[0, (n+1)a] \cup [b,na+c] \subseteq E_n + \{0, d\}$ if and only if 
$$n \ge \max\left\{\dfrac {b-1}{a}, \dfrac {e-c+a-1}{a}\right\}.$$

By symmetry, we have $[(n+1)e, (n+1)d] \cup[ne+b,nd+c]\subseteq E_n + \{0, d\}$ if and only if 
$$n \ge \max\left\{\dfrac {d-c-1}{d-e}, \dfrac {d-e+b-a-1}{d-e}\right\}.$$

Therefore, (6.9) and hence the condition $E_{n+1} = E_n + \{0, d\}$ are satisfied if and only if
$$n \ge \max\left\{\dfrac {b-1}{a}, \dfrac {e-c+a-1}{a},\dfrac {d-c-1}{d-e}, \dfrac {d-e+b-a-1}{d-e} \right\}.$$
By Lemma \ref{interval}, this implies 
$$
r_Q(R_+)= \max\left\{\left\lceil\dfrac {b-1}{a}\right\rceil, \left\lceil\dfrac {e-c+a-1}{a}\right\rceil, \left\lceil\dfrac {d-c-1}{d-e}\right\rceil, \left\lceil\dfrac {d-e+b-a-1}{d-e}\right\rceil \right\}.
$$
\end{proof}

We are not able to give an answer to the following question.

\begin{quest}
Is it true that $\reg(R) = r_Q(R_+)$ under the assumption of Theorem \ref{6}?
\end{quest}

\noindent {\bf Acknowledgement}. This paper is partially supported by grant 101.04-2019.313 of Vietnam National Foundation for Science and Technology Development.


\end{document}